\documentclass[11pt]{article}

\usepackage[utf8]{inputenc}

\usepackage{mathtools}
\usepackage{enumerate}
\usepackage{graphicx}
\usepackage{mathrsfs}
\usepackage{amssymb}
\usepackage{amsmath}
\usepackage{amsthm}
\usepackage{bbm}

\usepackage{algorithm}
\usepackage{algpseudocode}

\usepackage{soul,color}

\newtheorem{theorem}{Theorem}

\makeatletter
\newcommand*\bigcdot{\mathpalette\bigcdot@{.5}}
\newcommand*\bigcdot@[2]{\mathbin{\vcenter{\hbox{\scalebox{#2}{$\m@th#1\bullet$}}}}}
\makeatother

\title{On the Conditional Expectation of Mean Shifted Gaussian Distributions}
\author{Kananart Kuwaranancharoen
\thanks{The author is with the School of Electrical and Computer Engineering at Purdue University.  Email: {\tt kkuwaran@purdue.edu}.}}
\begin{document}

\maketitle

\begin{abstract}
    In this paper, we consider a property of univariate Gaussian distributions namely \textit{conditional expectation shift} (or \textit{centroid shift}). Specifically, we compare two Gaussian distributions in which they differ only in their means. Equivalently, we can view this situation as one of the distribution is shifted to the right. These two distributions are
    conditioned on the same event in which the realizations fall in the \textit{right interval} or \textit{left interval}. 
    We show that if a Gaussian distribution is shifted to the right while the conditioning event remains the same then the conditional expectation is shifted to the right concurrently.
\end{abstract}

\section{Problem}

Define the \textit{centroid} of a function $f: \mathbb{R} \to \mathbb{R}$ under a support $\mathcal{S} \subseteq \mathbb{R}$ as
\begin{align}
    \langle f \rangle_{\mathcal{S}} 
    \triangleq \frac{1}{\int_{\mathcal{S}} f(x) dx} \cdot \int_{\mathcal{S}} x f(x) dx
    = \frac{1}{\int f(x) \cdot \mathbbm{1}_{\mathcal{S}}(x) dx} \cdot
    \int x f(x) \cdot \mathbbm{1}_{\mathcal{S}}(x)  dx
    .  \label{def: average function}
\end{align}
where the \textit{indicator function} $\mathbbm{1}_{\mathcal{S}}: \mathbb{R} \to \{ 0, 1 \}$ is defined as
\begin{align*}
    \mathbbm{1}_{\mathcal{S}}(x) =
    \begin{cases}
    1 \quad \text{if} \;\; x \in \mathcal{S}, \\
    0 \quad \text{if} \;\; x \notin \mathcal{S}.
    \end{cases}
\end{align*}
However, if we consider $f_{\boldsymbol{X}}$ which is the \textit{probability density function} (pdf) of a random variable $\boldsymbol{X}$, the above definition coincides with the conditional expectation given that $ \boldsymbol{X} \in \mathcal{S}$ as shown below:
\begin{align*}
    \mathbb{E}[ \boldsymbol{X} | \boldsymbol{X} \in \mathcal{S} ]
    &= \int x f_{\boldsymbol{X}} (x | \boldsymbol{X} \in \mathcal{S}) dx \\
    &= \frac{1}{P( \boldsymbol{X} \in \mathcal{S})} \cdot \int_{\mathcal{S}} x f(x) dx \\
    &= \frac{1}{\int_{\mathcal{S}} f_{\boldsymbol{X}}(x) dx} \cdot \int_{\mathcal{S}} x f_{\boldsymbol{X}}(x) dx \\
    &= \langle f_{\boldsymbol{X}} \rangle_{\mathcal{S}}.
\end{align*}

In particular, we consider two univariate Gaussian random variables $\boldsymbol{X}$ and $\boldsymbol{Y}$ in which their pdfs differ only in their means namely $\mathcal{N}(\mu, \sigma^2)$ and $\mathcal{N}(\mu+h, \sigma^2)$, respectively, i.e., the variance of distributions is the same. We are interested in the relationship between conditional expectation of these two random variables. Specifically, given that $\boldsymbol{X} \in \mathcal{S}$ and $\boldsymbol{Y} \in \mathcal{S}$ where $\mathcal{S} = (-\infty, \ell] \cup [u, \infty)$ for some constant $u > \ell$, does it hold true that the conditional expectation of $\boldsymbol{Y}$ is strictly greater than that of $\boldsymbol{X}$ when $h > 0$? In other words, we want to show that when $h>0$, $\langle f_{\boldsymbol{Y}} \rangle_{\mathcal{S}} > \langle f_{\boldsymbol{X}} \rangle_{\mathcal{S}}$ where $f_{\bigcdot}$ is the pdf of the random variable ``$\boldsymbol{\bigcdot}$". Hence the name ``conditional expectation shift" or ``controid shift". To clarify the concept, we provide an example illustrated in Figure~\ref{fig: centroid}. In the next section, we provide the theorem regrading to this question and also its proof.

\begin{figure}[t]
\includegraphics[width=10cm]{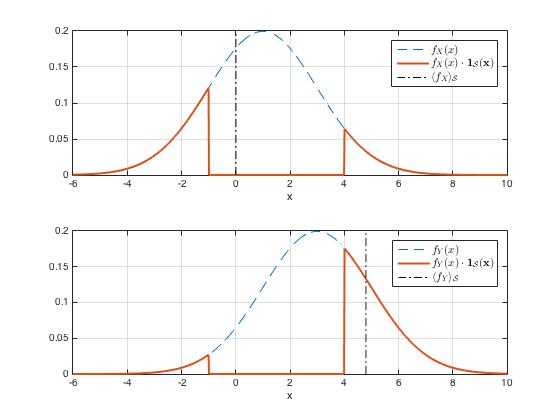}
\centering
\caption{In this example, we use $\mu = 1$, $\sigma^2 = 4$, and $h = 2$, i.e., $\boldsymbol{X} \sim \mathcal{N}(1,4)$ and $\boldsymbol{Y} \sim \mathcal{N}(3,4)$. Also, we set $\ell = -1$ and $u=4$ which means that $\mathcal{S} = (-\infty, -1] \cup [4,\infty)$. We also calculate the centroid of both functions; $f_{\boldsymbol{X}}(x) \cdot \mathbbm{1}_{\mathcal{S}}(x)$ and $f_{\boldsymbol{Y}}(x) \cdot \mathbbm{1}_{\mathcal{S}}(x)$. The centroids are $\langle f_{\boldsymbol{X}} \rangle_{\mathcal{S}} \approx 0.0025$ and $\langle f_{\boldsymbol{Y}} \rangle_{\mathcal{S}} \approx 4.7995$. In this case, we can see that $\langle f_{\boldsymbol{Y}} \rangle_{\mathcal{S}} > \langle f_{\boldsymbol{X}} \rangle_{\mathcal{S}}$.}
\label{fig: centroid}
\end{figure}

\section{Main Theorem and Proof}

\begin{theorem}
Suppose 
\begin{align}
    f(x) = \frac{1}{\sqrt{2 \pi \sigma^2}} \exp{ \Big( - \frac{1}{2 \sigma^2} (x- \mu)^2 \Big)} \label{eqn: Gaussian pdf}
\end{align}
with $\mu \in \mathbb{R}$ and $\sigma \in \mathbb{R}_{>0}$.
Let $g(x) = f(x-h)$ and $\mathcal{S} = \mathbb{R} \setminus (\ell, u)$. 
Then, for all $h \in \mathbb{R}_{>0}$, and $\ell, u \in \mathbb{R}$ such that $u > \ell$,
\begin{align}
    \langle g \rangle_{\mathcal{S}} > \langle f \rangle_{\mathcal{S}}.  \label{eqn: averagae ineq}
\end{align}
\end{theorem}

\begin{proof}
First, we will show that it is necessary and sufficient to consider 
\begin{align*}
    \langle \hat{g} \rangle_{\hat{\mathcal{S}}} > \langle \hat{f} \rangle_{\hat{\mathcal{S}}} 
\end{align*}
instead of the inequality \eqref{eqn: averagae ineq}, where
\begin{align*}
    \hat{f}(x) = \frac{1}{\sqrt{2 \pi}} \exp{ \Big( - \frac{1}{2} x^2 \Big)},
\end{align*}
i.e., the pdf of a random variable $\hat{\boldsymbol{X}} \sim \mathcal{N}(0,1)$,
$\hat{g}(x) = \hat{f}(x - \hat{h})$ with $\hat{h} = \frac{h}{\sigma}$,
and $\hat{\mathcal{S}} = \mathbb{R} \setminus (\hat{\ell}, \hat{u})$ with $\hat{u} = \frac{u-\mu}{\sigma}$ and $\hat{\ell} = \frac{\ell-\mu}{\sigma}$.

Consider the term $\int_{\mathcal{{S}}} x g(x) dx$ as follows:
\begin{align*}
    \int_{\mathcal{{S}}} x g(x) dx 
    = \int_{\mathcal{S}} \frac{x}{\sqrt{2 \pi \sigma^2}} \exp{ \Big( -\frac{1}{2 \sigma^2} (x-h-\mu)^2 \Big) } dx.
\end{align*}
By substituting $x = \sigma t + \mu$, we get 
\begin{align*}
    \int_{\mathcal{{S}}} x g(x) dx 
    &= \int_{\hat{\mathcal{S}}} (\sigma t + \mu) \cdot \frac{1}{\sqrt{2 \pi}} \exp{ \Big( -\frac{1}{2} \Big( t - \frac{h}{\sigma} \Big)^2 \Big)  } dt \\
    &= \sigma \int_{\hat{\mathcal{S}}} \frac{t}{\sqrt{2 \pi}} \exp{ \Big( -\frac{1}{2}(t - \hat{h})^2 \Big) } dt + \mu \int_{\hat{\mathcal{S}}} \frac{1}{\sqrt{2 \pi}} \exp{ \Big( -\frac{1}{2}(t - \hat{h})^2 \Big) } dt \\
    &= \sigma \int_{\hat{\mathcal{S}}} t \hat{g}(t) dt + \mu  \int_{\hat{\mathcal{S}}}  \hat{g}(t) dt.
\end{align*}
Then, consider the term $\int_{\mathcal{{S}}} g(x) dx$ as follows:
\begin{align*}
    \int_{\mathcal{{S}}} g(x) dx 
    = \int_{\mathcal{S}} \frac{1}{\sqrt{2 \pi \sigma^2}} \exp{ \Big( -\frac{1}{2 \sigma^2} (x-h-\mu)^2 \Big) } dx.
\end{align*}
Again, by substituting $x = \sigma t + \mu$, we get
\begin{align*}
    \int_{\mathcal{{S}}} g(x) dx 
    = \int_{\hat{\mathcal{S}}} \frac{1}{\sqrt{2 \pi}} \exp{ \Big( -\frac{1}{2} \Big( t - \frac{h}{\sigma} \Big)^2 \Big)  } dt 
    = \int_{\hat{\mathcal{S}}} \frac{1}{\sqrt{2 \pi}} \exp{ \Big( -\frac{1}{2} ( t - \hat{h} )^2 \Big)  } dt 
    = \int_{\hat{\mathcal{S}}}  \hat{g}(t) dt.
\end{align*}
Therefore, using the definition of the centroid of a function \eqref{def: average function}, we obtain that
\begin{align}
    \langle g \rangle_{\mathcal{S}}
    = \frac{1}{\int_{\mathcal{S}} g(x) dx} \cdot \int_{\mathcal{S}} x g(x) dx 
    = \mu + \frac{\sigma}{\int_{\hat{\mathcal{S}}} \hat{g}(x) dx} \cdot \int_{\hat{\mathcal{S}}} x \hat{g}(x) dx 
    = \mu + \sigma \langle \hat{g} \rangle_{\hat{\mathcal{S}}}. \label{eqn: average g}
\end{align}
In fact, the above analysis holds for any $h \in \mathbb{R}$. By letting $h=0$, we also obtain that
\begin{align}
    \langle f \rangle_{\mathcal{S}} = \mu + \sigma \langle \hat{f} \rangle_{\hat{\mathcal{S}}} . \label{eqn: average f} 
\end{align}
Since $\sigma \in \mathbb{R}_{>0}$, substituting the expressions \eqref{eqn: average g} and \eqref{eqn: average f} into \eqref{eqn: averagae ineq} yields
\begin{align*}
    \langle g \rangle_{\mathcal{S}} > \langle f \rangle_{\mathcal{S}}
    \quad \Longleftrightarrow \quad 
    \langle \hat{g} \rangle_{\hat{\mathcal{S}}} > \langle \hat{f} \rangle_{\hat{\mathcal{S}}}
\end{align*}
which proves the claim.

The claim suggests that without loss of generality, we can consider 
\begin{align}
    f(x) = \frac{1}{\sqrt{2 \pi}} \exp{ \Big( - \frac{1}{2} x^2 \Big)}, \label{eqn: N Gaussian pdf}
\end{align}
i.e., the standard Gaussian distribution, instead of the one in \eqref{eqn: Gaussian pdf}. To simplify the notation, starting from here we will stick to the expression of $f$ in \eqref{eqn: N Gaussian pdf}. Before moving on to the next claim, consider
\begin{align}
    \int_{\mathcal{S}} x g(x) dx
    &= \int_{\mathcal{S}} \frac{x}{\sqrt{2 \pi}} \exp{ \Big( -\frac{1}{2} (x-h)^2 \Big) } dx \nonumber \\
    &= h \int_{\mathcal{S}'} \frac{1}{\sqrt{2 \pi}} \exp{ \Big( -\frac{1}{2} x^2 \Big) } dx
    +  \int_{\mathcal{S}'} \frac{x}{\sqrt{2 \pi}} \exp{ \Big( -\frac{1}{2} x^2 \Big) } dx  \nonumber\\
    &= h \int_{\mathcal{S}'} f(x) dx +  \int_{\mathcal{S}'} \frac{x}{\sqrt{2 \pi}} \exp{ \Big( -\frac{1}{2} x^2 \Big) } dx  \label{eqn: numerator g}
\end{align}
where $\mathcal{S}' = \mathbb{R} \setminus (\ell -h, u-h)$. And,
\begin{align}
    \int_{\mathcal{S}} g(x) dx
    = \int_{\mathcal{S}} \frac{1}{\sqrt{2 \pi}} \exp{ \Big( -\frac{1}{2} (x-h)^2 \Big) } dx
    = \int_{\mathcal{S}'} \frac{1}{\sqrt{2 \pi}} \exp{ \Big( -\frac{1}{2} x^2 \Big) } dx
    = \int_{\mathcal{S}'} f(x) dx. \label{eqn: denominator g}
\end{align}
However, we can simplify the last term on the RHS of \eqref{eqn: numerator g} as follows:
\begin{align}
    \int_{\mathcal{S}'} \frac{x}{\sqrt{2 \pi}} \exp{ \Big( -\frac{1}{2} x^2 \Big) } dx
    &= \int_{-\infty}^{\ell-h} \frac{x}{\sqrt{2 \pi}} \exp{ \Big( -\frac{1}{2} x^2 \Big) } dx 
    + \int_{u-h}^{\infty} \frac{x}{\sqrt{2 \pi}} \exp{ \Big( -\frac{1}{2} x^2 \Big) } dx \nonumber \\
    &= \frac{1}{\sqrt{2 \pi}} \Big[  -\exp{ (-x) } \Big]_{x=\infty}^{x= \frac{1}{2} (\ell -h)^2}
    + \frac{1}{\sqrt{2 \pi}} \Big[ -\exp{ (-x) } \Big]_{x=\frac{1}{2} (u-h)^2}^{x=\infty} \nonumber \\
    &= \frac{1}{\sqrt{2 \pi}} \exp{ \Big( -\frac{1}{2} (u-h)^2 \Big) } 
    - \frac{1}{\sqrt{2 \pi}} \exp{ \Big( -\frac{1}{2} (\ell-h)^2 \Big) } \nonumber \\
    &= f(u-h) - f(\ell-h).  \label{eqn: integral simplify}
\end{align}
Using the equations \eqref{eqn: numerator g}, \eqref{eqn: denominator g} and \eqref{eqn: integral simplify}, we can write
\begin{align}
    \langle g \rangle_{\mathcal{S}}
    &= \frac{1}{\int_{\mathcal{S}} g(x) dx} \cdot \int_{\mathcal{S}} x g(x) dx \nonumber \\
    &= \frac{1}{\int_{\mathcal{S}'} f(x) dx} \bigg[ h \int_{\mathcal{S}'} f(x) dx + \big( f(u-h) - f(\ell-h) \big) \bigg] \nonumber \\
    &= h + \frac{f(u-h) - f(\ell-h)}{ \int_{\mathcal{S}'} f(x) dx }. \label{eqn: average g_1}
\end{align}
To simplify the term $\int_{\mathcal{S}'} f(x) dx$ in \eqref{eqn: average g_1}, we define $\Phi: \mathbb{R} \to [0,1]$ to be the \textit{cumulative distribution function} (cdf) of the standard Gaussian distribution $f$, i.e.,
\begin{align*}
    \Phi(x) \triangleq \int_{-\infty}^{x} \frac{1}{\sqrt{2 \pi}} \exp{ \Big( - \frac{1}{2} t^2 \Big)} dt
    = \int_{-\infty}^{x} f(t) dt
\end{align*}
and define $\mathcal{Q}: \mathbb{R} \to [0,1]$ to be the \textit{tail distribution function} of the standard Gaussian distribution $f$, i.e.,
\begin{align*}
    \mathcal{Q}(x) \triangleq \int_{x}^{\infty} \frac{1}{\sqrt{2 \pi}} \exp{ \Big( - \frac{1}{2} t^2 \Big)} dt
    = \int_{x}^{\infty} f(t) dt.
\end{align*}
Note that $\Phi(x) + \mathcal{Q}(x) = 1$ and $\mathcal{Q}(-x) = \Phi(x)$ for all $x \in \mathbb{R}$. With these definitions in hand, we can rewrite the term $\int_{\mathcal{S}'} f(x) dx$ in \eqref{eqn: average g_1} as
\begin{align*}
    \int_{\mathcal{S}'} f(x) dx 
    = \int_{-\infty}^{\ell-h} f(x) dx + \int_{u-h}^{\infty} f(x) dx 
    = \Phi(\ell - h) + \mathcal{Q}(u-h).
\end{align*}
Therefore, the equation \eqref{eqn: average g_1} becomes
\begin{align*}
    \langle g \rangle_{\mathcal{S}}
    = h + \frac{f(u-h) - f(\ell-h)}{ \mathcal{Q}(u-h) + \Phi(\ell - h) }.
\end{align*}
Since $g(x) = f(x-h)$, we have
\begin{align*}
    \langle f \rangle_{\mathcal{S}}
    = \langle g \rangle_{\mathcal{S}} \Big|_{h=0}
    = \frac{f(u) - f(\ell)}{ \mathcal{Q}(u) + \Phi(\ell) }.
\end{align*}

For fixed $u$, $\ell \in \mathbb{R}$ (with $u > \ell$), define a function $\Psi: \mathbb{R} \to \mathbb{R}$ to be
\begin{align}
    \Psi(h) \triangleq h + \frac{f(u-h) - f(\ell-h)}{ \mathcal{Q}(u-h) + \Phi(\ell - h) }.  \label{def: Psi}
\end{align}
Notice that $\Psi(h) = \langle g \rangle_{\mathcal{S}}$ and $\Psi(0) = \langle f \rangle_{\mathcal{S}}$.
Next, we want to show that for all $h \in \mathbb{R}$,
\begin{align*}
    \frac{d}{dh} \Psi(h) > 0 .
\end{align*}
The derivative can be expressed as
\begin{multline}
    \frac{d}{dh} \Psi(h) = 1 + \frac{1}{ \big[ \mathcal{Q}(u-h) + \Phi(\ell - h) \big]^2} \cdot 
    \Big\{ \big[ \mathcal{Q}(u-h) + \Phi(\ell - h) \big] \frac{d}{dh} \big[ f(u-h) - f(\ell-h) \big] \\
    - \big[ f(u-h) - f(\ell-h) \big] \frac{d}{dh} \big[ \mathcal{Q}(u-h) + \Phi(\ell - h) \big] \Big\}.  \label{eqn: diff Psi}
\end{multline}
Suppose $c \in \mathbb{R}$ is a constant. We will simplify some expressions as follows.
\begin{align}
    \frac{d}{dh} f(c-h) 
    &= \frac{1}{\sqrt{2 \pi}} \frac{d}{dh} \exp{ \Big( -\frac{1}{2} (c-h)^2 \Big) }  \nonumber \\
    &= (c-h) \cdot \frac{1}{\sqrt{2 \pi}} \exp{ \Big( -\frac{1}{2} (c-h)^2 \Big) } \nonumber \\
    &= (c-h) f(c-h). \label{eqn: diff f}
\end{align}
Since $\Phi(x)$ is the cdf of the standard Gaussian distribution and $\mathcal{Q}(x) = 1 - \Phi(x)$ for any $x \in \mathbb{R}$, we have
\begin{align}
    \frac{d}{dh} \Phi(c - h) = - f(c-h)
    \quad \text{and} \quad 
    \frac{d}{dh} \mathcal{Q}(c - h) = f(c-h).  \label{eqn: diff Phi}
\end{align}
Applying the expressions \eqref{eqn: diff f} and \eqref{eqn: diff Phi} to \eqref{eqn: diff Psi}, we obtain that
\begin{align*}
    \frac{d}{dh} \Psi(h) 
    &= 1 + \frac{ \big[ (u-h)f(u-h) - (\ell-h)f(\ell-h) \big] \big[ \mathcal{Q}(u-h) + \Phi(\ell - h) \big] - \big[ f(u-h)-f(\ell-h) \big]^2 }{ \big[ \mathcal{Q}(u-h) + \Phi(\ell - h) \big]^2 } \\
    &= \frac{1}{ \big[ \mathcal{Q}(u-h) + \Phi(\ell - h) \big]^2} \cdot \Big\{ \big[ (u-h)f(u-h) - (\ell-h)f(\ell-h) \big] \big[ \mathcal{Q}(u-h) + \Phi(\ell - h) \big] \\
    & \qquad\qquad\qquad\qquad\qquad 
    + \big[ \mathcal{Q}(u-h) + \Phi(\ell - h) \big]^2 - \big[f(u-h)-f(\ell-h) \big]^2 \Big\}.
\end{align*}
Since $\big[ \mathcal{Q}(u-h) + \Phi(\ell - h) \big]^2 > 0$ for all $h \in \mathbb{R}$, we can consider only the term in the curly bracket. 
Define a function $\Omega: \mathbb{R} \times \mathbb{R} \to \mathbb{R}$ to be
\begin{align}
    \Omega(x_1, x_2) \triangleq \big( x_1 f(x_1) - x_2 f(x_2) \big) \big( \mathcal{Q}(x_1) + \Phi(x_2) \big) + \big( \mathcal{Q}(x_1) + \Phi(x_2) \big)^2 - \big( f(x_1) - f(x_2) \big)^2.  \label{def: Omega}
\end{align}
If we can show that $\Omega(x_1, x_2) > 0$ for all $x_1, x_2 \in \mathbb{R}$, this implies $\frac{d}{dh} \Psi(h) > 0$ for all $h \in \mathbb{R}$ which is our desired result.

Before we proceed, let's derive a lower bound of the ratios $\frac{\mathcal{Q}(x)}{2 f(x)}$ and $\frac{\Phi(x)}{2 f (x)}$.
Consider the inequality \cite[Formula~7.1.13]{abramowitz1948handbook}
\begin{align}
    \exp{ (x^2) } \int_x^{\infty} \exp{( -t^2 )} dt 
    > \frac{1}{x + \sqrt{x^2+2}}  \label{eqn: erf ineq}
\end{align}
which holds for all $x \in \mathbb{R}$.
Using the substitution $t = \frac{v}{\sqrt{2}}$, we have
\begin{align}
    \exp{ (x^2) } \int_x^{\infty} \exp{( -t^2 )} dt
    = \frac{1}{\sqrt{2}}\exp{ (x^2) } \int_{ \sqrt{2} x}^{\infty} \exp{ \Big( - \frac{1}{2}v^2 \Big)} dv  \label{eqn: sub erf}
\end{align}
Combining \eqref{eqn: erf ineq} and \eqref{eqn: sub erf} together and using substitution $x = \frac{z}{\sqrt{2}}$ to obtain
\begin{align*}
    \frac{1}{\sqrt{2}} \exp{ \Big( \frac{1}{2}z^2 \Big)} \int_{z}^{\infty} \exp{ \Big( - \frac{1}{2}v^2 \Big)} dv 
    > \frac{1}{\frac{z}{\sqrt{2}} + \sqrt{\frac{z^2}{2}+2}}
\end{align*}
which is equivalent to (by multiplying both sides by $\frac{1}{\sqrt{2}}$ and renaming the variables)
\begin{align}
    \frac{\mathcal{Q}(x)}{2 f(x)} 
    = \frac{ \frac{1}{\sqrt{2 \pi}} \int_{x}^{\infty} \exp{ \big( - \frac{1}{2} t^2 \big)} dt }{ 2 \cdot \frac{1}{\sqrt{2 \pi}} \exp{ \big( - \frac{1}{2} x^2 \big)}  }
    > \frac{1}{x + \sqrt{x^2 + 4}}.   \label{eqn: Q/f ineq}
\end{align}
In order to get the lower bound of $\frac{\Phi(x)}{2 f (x)}$, we substitute $x = -z$ into \eqref{eqn: Q/f ineq} to get
\begin{align*}
    \frac{\mathcal{Q}(-z)}{2 f(-z)} > \frac{1}{-z + \sqrt{z^2 + 4}}.
\end{align*}
Using the properties $\mathcal{Q}(-z) = \Phi(z)$ and $f(-z) = f(z)$ for all $z \in \mathbb{R}$ (and renaming the variable) yields
\begin{align}
    \frac{\Phi(x)}{2 f(x)} > \frac{1}{-x + \sqrt{x^2 + 4}}. \label{eqn: Phi/f ineq}
\end{align}

From the lower bound \eqref{eqn: Q/f ineq}, for all $x_1 \in \mathbb{R}$, we have
\begin{align}
    &\frac{\mathcal{Q}(x_1)}{2 f(x_1)}
    > \frac{1}{x_1 + \sqrt{x_1^2 +4}}
    = \frac{1}{4} \big( -x_1 + \sqrt{x_1^2+4} \big) \nonumber \\
    \Longleftrightarrow \quad 
    &\big( 2 \mathcal{Q}(x_1) + x_1 f(x_1) \big) - \sqrt{x_1^2+4} \cdot f(x_1) > 0. \label{eqn: Q ineq}
\end{align}
On the other hand, from the lower bound \eqref{eqn: Phi/f ineq}, for all $x_2 \in \mathbb{R}$, we have
\begin{align}
    &\frac{\Phi(x_2)}{2 f (x_2)} 
    > \frac{1}{- x_2 + \sqrt{x_2^2 +4}}
    = \frac{1}{4} \big( x_2 + \sqrt{x_2^2+4} \big) \nonumber \\
    \Longleftrightarrow \quad 
    &\big( 2 \Phi(x_2) - x_2 f(x_2) \big) - \sqrt{x_2^2+4} \cdot f(x_2) > 0.  \label{eqn: Phi ineq}
\end{align}
Summing the inequalities \eqref{eqn: Q ineq} and \eqref{eqn: Phi ineq} together yields
\begin{align}
    2 \big[ \mathcal{Q}(x_1) + \Phi(x_2) \big] + \big[ x_1 f(x_1) - x_2 f(x_2) \big] 
    - \Big[ \sqrt{x_1^2+4} \cdot f(x_1) + \sqrt{x_2^2+4} \cdot f(x_2) \Big] > 0. \label{eqn: Q Phi ineq}
\end{align}
Since $x_1 + \sqrt{x_1^2+4} > 0$ for all $x_1 \in \mathbb{R}$ and $-x_2 + \sqrt{x_2^2+4} > 0$ for all $x_2 \in \mathbb{R}$, we have
\begin{align*}
    &\Big( x_1 + \sqrt{x_1^2+4} \Big) f(x_1) + \Big( - x_2 + \sqrt{x_2^2+4} \Big) f(x_2) > 0 \\
    \Longrightarrow \quad 
    &  2 \big[ \mathcal{Q}(x_1) + \Phi(x_2) \big] + \Big( x_1 + \sqrt{x_1^2+4} \Big) f(x_1) + \Big( - x_2 + \sqrt{x_2^2+4} \Big) f(x_2) > 0 \\
    \Longleftrightarrow \quad
    & 2 \big[ \mathcal{Q}(x_1) + \Phi(x_2) \big] + \big[ x_1 f(x_1) - x_2 f(x_2) \big] 
    + \Big[ \sqrt{x_1^2+4} \cdot f(x_1) + \sqrt{x_2^2+4} \cdot f(x_2) \Big] > 0
\end{align*}
Multiplying both sides of \eqref{eqn: Q Phi ineq} by the term on the LHS of the inequality above yields
\begin{align*}
    \Big\{ 2 \big[ \mathcal{Q}(x_1) + \Phi(x_2) \big] + \big[ x_1 f(x_1) - x_2 f(x_2) \big] \Big\}^2
    - \Big[ \sqrt{x_1^2+4} \cdot f(x_1) + \sqrt{x_2^2+4} \cdot f(x_2) \Big]^2 > 0.
\end{align*}
Simplify the above inequality to get
\begin{multline*}
    \big[ x_1 f(x_1) - x_2 f(x_2) \big] \big[ \mathcal{Q}(x_1) + \Phi(x_2) \big] + \big[ \mathcal{Q}(x_1) + \Phi(x_2) \big]^2 \\
    + \frac{1}{4} \Big\{ \big[ x_1 f(x_1) - x_2 f(x_2) \big]^2
    - \Big[ \sqrt{x_1^2+4} \cdot f(x_1) + \sqrt{x_2^2+4} \cdot f(x_2) \Big]^2 \Big\} > 0.
\end{multline*}
Using the definition of $\Omega(x_1, x_2)$ in \eqref{def: Omega}, we can rewrite the above inequality as 
\begin{multline}
    \Omega(x_1, x_2) + \big[ f(x_1) - f(x_2) \big]^2 \\
    + \frac{1}{4} \Big\{ \big[ x_1 f(x_1) - x_2 f(x_2) \big]^2
    - \Big[ \sqrt{x_1^2+4} \cdot f(x_1) + \sqrt{x_2^2+4} \cdot f(x_2) \Big]^2 \Big\} > 0. \label{eqn: Omega+}
\end{multline}
We will show that 
\begin{align}
    \big[ f(x_1) - f(x_2) \big]^2 
    + \frac{1}{4} \Big\{ \big[ x_1 f(x_1) - x_2 f(x_2) \big]^2
    - \Big[ \sqrt{x_1^2+4} \cdot f(x_1) + \sqrt{x_2^2+4} \cdot f(x_2) \Big]^2 \Big\} < 0 \label{eqn: f1 f2 ineq}
\end{align}
which implies that $\Omega(x_1, x_2) > 0$ from \eqref{eqn: Omega+} as desired.

Consider the second term on the LHS of \eqref{eqn: f1 f2 ineq} as follows:
\begin{align*}
    &\frac{1}{4} \Big\{ \big[ x_1 f(x_1) - x_2 f(x_2) \big]^2
    - \Big[ \sqrt{x_1^2+4} \cdot f(x_1) + \sqrt{x_2^2+4} \cdot f(x_2) \Big]^2 \Big\} \\
    = \; &\frac{1}{4} \Big\{ \big[ x_1^2 f^2(x_1) + x_2^2 f^2(x_2) - 2 x_1 x_2 f(x_1) f(x_2) \big] \\
    & \qquad\qquad\qquad\qquad
    - \Big[ (x_1^2 + 4) f^2(x_1) + (x_2^2 + 4) f^2(x_2) 
    + 2 \sqrt{x_1^2+4} \sqrt{x_2^2+4} \cdot f(x_1) f(x_2) \Big] \Big\} \\
    = \; &  -f^2(x_1) - f^2(x_2) 
    - \frac{1}{2} x_1 x_2 f(x_1) f(x_2)
    - \frac{1}{2} \sqrt{x_1^2+4} \sqrt{x_2^2+4} \cdot f(x_1) f(x_2). 
\end{align*}
Thus, the term on the LHS of \eqref{eqn: f1 f2 ineq} becomes
\begin{align*}
    &\big[ f(x_1) - f(x_2) \big]^2 
    + \frac{1}{4} \Big\{ \big[ x_1 f(x_1) - x_2 f(x_2) \big]^2
    - \Big[ \sqrt{x_1^2+4} \cdot f(x_1) + \sqrt{x_2^2+4} \cdot f(x_2) \Big]^2 \Big\} \\
    = \; & - 2 f(x_1) f(x_2) - \frac{1}{2} x_1 x_2 f(x_1) f(x_2)
    - \frac{1}{2} \sqrt{x_1^2+4} \sqrt{x_2^2+4} \cdot f(x_1) f(x_2) \\
    = \; & - \Big[ 2 + \frac{1}{2} \Big( x_1 x_2 + \sqrt{x_1^2+4} \sqrt{x_2^2+4} \Big) \Big] f(x_1) f(x_2).
\end{align*}
Since $f(x_1) f(x_2) > 0$ for all $x_1, x_2 \in \mathbb{R}$, it remains to show that
\begin{align*}
    x_1 x_2 + \sqrt{x_1^2+4} \sqrt{x_2^2+4} > 0
\end{align*}
for all $x_1, x_2 \in \mathbb{R}$.
Consider the positive expression for all $x_1, x_2 \in \mathbb{R}$ below:
\begin{align*}
    4 x_1^2 + 4 x_2^2 + 16 &> 0 \\
    \Longleftrightarrow \quad
    x_1^2 x_2^2 +  4 x_1^2 + 4 x_2^2 + 16 &> x_1^2 x_2^2 \\
    \Longleftrightarrow \quad
    (x_1^2 + 4) (x_2^2 + 4) &> x_1^2 x_2^2 \\
    \Longleftrightarrow \quad
    \sqrt{x_1^2+4} \sqrt{x_2^2+4} &> | x_1 x_2 |  \\
    \Longrightarrow \quad
    \sqrt{x_1^2+4} \sqrt{x_2^2+4} &> - x_1 x_2 \\
    \Longleftrightarrow \quad
    \sqrt{x_1^2+4} \sqrt{x_2^2+4} + x_1 x_2 &> 0.
\end{align*}
Again, this means that the inequality \eqref{eqn: f1 f2 ineq} holds for all $x_1, x_2 \in \mathbb{R}$. Thus, $\Omega(x_1, x_2) > 0$ for all $x_1, x_2 \in \mathbb{R}$ as discussed earlier (where $\Omega(x_1, x_2)$ is defined in \eqref{def: Omega}). Consequently, we have proved the claim $\frac{d}{dh} \Psi(h) > 0$ for all $h \in \mathbb{R}$ or equivalently, $\Psi(h)$ is a strictly increasing function. In particular, for $h \in \mathbb{R}_{>0}$, we have
\begin{align*}
    \langle g \rangle_{\mathcal{S}} = \Psi(h) > \Psi(0) = \langle f \rangle_{\mathcal{S}}
\end{align*}
from the definition of $\Psi(h)$ in \eqref{def: Psi}, which completes the proof.
\end{proof}

\bibliographystyle{IEEEtran}
\bibliography{ref}

\end{document}